%% file: main.tex
\documentclass[9pt]{article}
\usepackage[left=3.8cm, right=3.8cm, top=3.8cm, bottom=2.25cm]{geometry}
\usepackage{amsmath,color}
\usepackage{amsfonts,amsthm}
\usepackage{amssymb,enumerate,enumitem,verbatim}
\usepackage[pdftex]{graphicx}
\usepackage[export]{adjustbox}
\usepackage{amsmath,setspace,scalefnt}
\usepackage[usenames,dvipsnames,svgnames,table]{xcolor}
\usepackage{amsfonts,amsthm}
\usepackage{amssymb,enumitem,verbatim}
\usepackage{graphicx}
\usepackage{tikz,caption,subcaption,thm-restate}
\usetikzlibrary{calc,decorations.pathmorphing}
\usepackage[titles]{tocloft}
\usepackage[symbol]{footmisc}
\usepackage{hyperref}
\usepackage{appendix}

\input{guff}

\newcommand{\claimproofstart}[1][Proof]{\begin{proof}[#1]
\renewcommand{\qedsymbol}{$\boxdot$}}
\newcommand\claimproofend{
\end{proof}
\renewcommand{\qedsymbol}{$\square$}}

\hypersetup{hidelinks}

\newcommand{\Alabel}[1]{%
  \textbf{A#1%
    \ifnum#1=1 *%
    \else\ifnum#1=7 *%
    \fi\fi
  }%
}

\makeatletter
\newcommand{\astar}[1]{%
  \expandafter\@astar\csname c@#1\endcsname
}
\newcommand{\@astar}[1]{%
  \@arabic#1%
  \ifnum#1=1 *%
  \else\ifnum#1=7 *%
  \fi\fi
}
\AddEnumerateCounter{\astar}{\@astar}{7*}
\makeatother

\title{Hamiltonicity of mildly pseudorandom regular graphs}
\author{Alp M\"uyesser\thanks{Mathematical Institute, University of Oxford. \textrm{email: alp.muyesser@maths.ox.ac.uk} }}

\begin{document}

\maketitle

\begin{abstract}
    We show that if an $(n,d,\lambda)$-graph satisfies $\lambda\leq (1-\delta)d$ and $d\gg  \delta^{-6}(\log n)^{3}$ for some $\delta>0$, then it is Hamiltonian. A qualitatively similar result was recently proven by Brada\v{c} and Janzer. Our proof here is shorter and gives better quantitative bounds.
\par  In our proof, as in earlier work of Ferber, Han, Mao, and Vershynin, we use a random matrix inequality to show that a mild spectral gap is typically preserved after randomly sampling an appropriate proportion of the vertices. This allows us to deduce that typical balanced bipartite subgraphs of pseudorandom graphs contain perfect matchings. To convert a collection of perfect matchings into a Hamilton cycle, we use a variant of the sorting network method.
\end{abstract}

\section{Introduction}
\par An $(n,d,\lambda)$-graph is an $n$-vertex $d$-regular graph where all non-trivial eigenvalues of the adjacency matrix have absolute value at most $\lambda$. Roughly speaking, the smaller $\lambda$ is compared to $d$, the more pseudorandom the underlying graph is and the better it expands. For example, it is well-known that $(n,d,\lambda)$-graphs have their edge expansion controlled by $d-\lambda$, see for example \cite[Theorem 2.4]{HooryLinialWigderson2006}.
\par An important milestone in the study of Hamiltonicity is the recent resolution of the Krivelevich--Sudakov conjecture \cite{krivelevich2003sparse} by Dragani\'c, Montgomery, Correia, Pokrovskiy, and Sudakov \cite{draganic2024hamiltonicity}. This conjecture, now a theorem, asserts that $(n,d,\lambda)$-graphs are Hamiltonian when $\lambda\leq d/C$, where $C$ is a sufficiently large constant. Although this is a widely useful criterion for Hamiltonicity, and covers many interesting cases such as random Cayley graphs with $C\log n$ generators, it is not powerful enough to deduce Hamiltonicity of various other graph families, such as Kneser graphs~\cite{Merino2025Kneser}, or recover the classical fact that hypercubes contain Hamilton cycles. 

\par The motivation of the present paper is to investigate to what extent the Krivelevich--Sudakov conjecture can be strengthened. In particular, we would like to determine the weakest possible assumption on an $(n,d,\lambda)$-graph that forces it to be Hamiltonian. Of course, any Hamiltonian graph has a near perfect matching (meaning a perfect matching that misses at most one vertex) and is connected. Krivelevich and Sudakov showed that $\lambda\leq d-2$ is sufficient to guarantee near perfect matchings \cite[Theorem 4.3]{krivelevich2006pseudo} and it is a standard exercise to check that $\lambda<d$ ensures connectivity. These two properties by themselves are not enough to establish Hamiltonicity; however, a widespread phenomenon in graph theory is that any non-trivial condition that implies that a graph has a near perfect matching and is connected also implies that the graph is Hamiltonian. 
\par The preceding discussion raises the natural question of whether all $(n,d,\lambda)$-graphs with $\lambda\leq d-2$ are Hamiltonian. The author is not aware of any counterexamples to such a conjecture, despite extensive search with modern AI tools. This should of course be taken with a grain of salt, as it remains very difficult to argue, in general, that a given graph is not Hamiltonian.
\par A more plausible conjecture is that a wider separation between $d$ and $\lambda$ should be sufficient for Hamiltonicity. 

\begin{conjecture}[Folklore]\label{conj:main} For any $\delta\in(0,1)$, there exists $d_0$ so that for every $d\geq d_0$, every $(n,d,\lambda)$-graph with $\lambda\leq (1-\delta)d$ is Hamiltonian.
\end{conjecture}
The author heard the above conjecture formulated by Ferber and stronger forms of it were given in print by Brada\v{c} and Janzer \cite{bradavc2026hamiltonicity} and by Montgomery~\cite{montgomery2026recent}. The author also heard the conjecture formulated as a question in various seminars, for example by Buci\'c and by Gowers. 
\par Brada\v{c} and Janzer made significant progress towards Conjecture~\ref{conj:main}, essentially resolving a strong version of it whenever the host graph has polylogarithmic degree. In particular, they showed the following.

\begin{theorem}[Brada\v{c}--Janzer]
    Let $G$ be an $(n,d,\lambda)$-graph, and suppose that $n$ is sufficiently large. If $\lambda\leq(1-\delta)d$
and $d\ge (\delta^{-1}\log n)^{10^9}$, then $G$ is Hamiltonian.
\end{theorem}

Our contribution in this paper is a shorter proof of the above that comes with better quantitative estimates.

\begin{theorem}\label{thm: main theorem}
Let $G$ be an $(n,d,\lambda)$-graph. If $\lambda\leq(1-\delta)d$
and $d\ge C  \cdot \delta^{-6}(\log n)^{3}$, where $\delta\in(0,1)$ and $C$ is a sufficiently large absolute constant, then $G$ is Hamiltonian.
\end{theorem}
We also record a combinatorial version of the above result that can be obtained by applying Cheeger's inequality (see, for example,
\cite[Lemmas~3.3 and~3.6]{bradavc2026hamiltonicity}). We say that a graph $G$ is $\varepsilon$-far from bipartite if at least
$\varepsilon e(G)$ edges must be deleted from $G$ in order to make it
bipartite.

\begin{corollary}
There is an absolute constant $C>0$ such that the following holds. Let $G$ be an $n$-vertex $d$-regular $\gamma$-expander, in the sense that
$
e_G(S,V(G)\setminus S)\ge \gamma d|S|
$
for every $S\subseteq V(G)$ with $1\le |S|\le 2n/3$. Suppose moreover that $G$ is $\varepsilon$-far from bipartite. If $n$ is sufficiently large and
$
d\ge C\varepsilon^{-12}\gamma^{-12}(\log n)^3,
$
then $G$ is Hamiltonian.
\end{corollary}

An analogous result can be obtained for bipartite graphs with one-sided spectral gap as well, as was done in \cite{bradavc2026hamiltonicity}, but we do not pursue this here.

\vspace{2mm}

\noindent \textbf{Remarks on the proof, acknowledgements, and AI declaration. }  An important ingredient in the proof is Lemma~\ref{lem:main}, which asserts, in a suitable weak sense, that random vertex subgraphs of $(n,d,\lambda)$-graphs retain pseudorandomness. Its proof uses only standard matrix concentration estimates, and the underlying phenomenon is not new. A substantially more sophisticated version was proved by Chen, Khanna, and Li~\cite{ChenKhannaLi22}, and by Assadi, Kapralov, and Yu~\cite{AssadiKapralovYu2023}, although with worse quantitative bounds.

The author is also indebted to the work of Ferber, Han, Mao, and Vershynin \cite{ferber2025hamiltonicity}, which first introduced him to these ideas. Their work proves a related sampling statement that yields partial progress towards the Krivelevich--Sudakov conjecture. The argument there has a constant factor loss as an artifact of the proof, which prevents a direct application to the mildly pseudorandom graphs considered in the present paper. As the author learned from Ferber, an unpublished variant of their argument avoids this loss at the expense of obtaining only a one-sided spectral estimate.

AI tools were used to help identify a short proof of the particular one-sided statement formulated here as Lemma~\ref{lem:main}. The underlying result is nevertheless attributed to Ferber, Han, Mao, and Vershynin. The formulation used in the present paper is weaker than what their methods appear capable of proving, but has the advantage of admitting a particularly short proof, consisting essentially of two applications of the matrix Bernstein inequality.
\par The remaining ingredients of the proof build on the work of other researchers. In particular, we build on the sorting network method, which goes back to work of H\"aggkvist and Thomasson \cite{HaggkvistThomason1995,HaggkvistThomason1997}. We use a version of this method inspired by our prior work with Hyde, Morrison, and Pavez--Sign\'e~\cite{hyde2025spanning}, and in particular, we use a simplification later discovered in our work with Bowtell, Montgomery, and Pokrovskiy \cite{bowtell2026proof} that uses Hamilton routers. Other important ingredients in the proof include Hall's marriage theorem and a very useful hypergraph generalisation thereof, due to Haxell \cite{Haxell1995}, which we use to find connections between prescribed pairs of vertices in expander graphs, as is standard in the area. The contribution of the present paper is assembling these tools in a convenient fashion.

\par The material is partially based upon work supported by the National Science Foundation under Grant No.~DMS-1928930, while the author was in residence at the Simons Laufer Mathematical Sciences Institute in Berkeley, California, during the semester of Spring 2025.

\section{Spectral results, subsampling expanders}

In this section, we record two important spectral facts. The first, Lemma~\ref{lemma: spectral cut density}, is quite standard, and says that spectral gap implies a combinatorial expansion condition. The second,  Lemma~\ref{lem:main}, states that randomly sampling rows and columns of a biadjacency matrix does not change the operator norm significantly, and follows by basic matrix inequalities, but appears to be not so well-known in extremal combinatorics. As stated in the acknowledgements, versions of such a statement were proved before in \cite{AssadiKapralovYu2023, ferber2025hamiltonicity, ChenKhannaLi22}. The proof we use here, a routine application of the Matrix Bernstein inequality, was suggested to us by an AI model.

We first record the following standard consequence of 
\cite[Corollary~4.7.4 and Theorem~2.8.1(iii)]{BrouwerHaemers2012} which allows us to translate spectral bounds into a combinatorial expansion property. We will not need it until the proof of Theorem~\ref{thm: main theorem}.

\vspace{2mm}
\noindent \textbf{Definition.} A graph $G$ is \textit{$p$-cut-dense} if
$
e_G(S,V(G)\setminus S)
\ge p|S||V(G)\setminus S|$ for every $S\subseteq V(G)$.

\begin{lemma}\label{lemma: spectral cut density}
    Let $G$ be a graph on $n\ge2$ vertices, and let
$\lambda_2(G)$ denote the second largest eigenvalue
of its adjacency matrix. If $\delta(G)>\lambda_2(G)$, then $G$ is
$
\frac{\delta(G) - \lambda_2(G)}{n}\text{-cut-dense}.
$
\end{lemma}

Given the previous lemma, to retain combinatorial expansion after random subsampling in a $(n,d,\lambda)$-graph, we need a way to deduce information about the spectrum of the randomly sampled subgraph. The following is the lemma that allows us to achieve this. The statement is essentially due to Ferber, Han, Mao, and Vershynin, as explained in the introduction, however, here, we give a short proof, discovered with AI assistance.
\begin{lemma}[Expansion inheritance under vertex subsampling]\label{lem:main}
For every $a>0$, the constant $C_a=30000(a+4)$ has the following property.
Let $G$ be an $(n,d,\lambda)$-graph with adjacency
matrix $A$, and suppose that $\lambda\leq (1-\delta)d$. Let $m\ge2$ be an integer, set $p=m/n$ and $D=pd$, and assume that $p\le\frac{\delta}{100}$ and $D\ge C_a\delta^{-2}\log n$. Choose $(U,V)$ uniformly among all ordered pairs of disjoint $m$-subsets
of $V(G)$, and let $K=A[U,V]$ be the biadjacency matrix of $G[U,V]$.
With probability at least $1-n^{-a}$, the following holds.
\begin{align}
 s_2(K)&\le\left(1-\frac{\delta}{2}\right)D.
 \label{eq:spectral-conclusion}
\end{align}
\end{lemma}

To prove the above, we use the following form of matrix Bernstein inequality. Here, $\lambda_{\max}(\cdot)$ returns the algebraically largest eigenvalue of a self-adjoint matrix and $X_j^2$
is just matrix multiplication. For a real matrix $F$, $\norm{F}$ denotes its
operator norm with respect to the Euclidean norm.

\begin{theorem}[Matrix Bernstein {\cite[Theorem~1.4]{Tropp2012}}]\label{thm: matrix berstein}
Let ${X_k}$ be a finite sequence of independent random self-adjoint matrices of dimension $d$. Suppose that $\mathbb E X_k=0$ and $\lambda_{\max}(X_k)\le R$ almost surely for every $k$. Then, for every $t\ge0$,

$$
\mathbb P\left\{\lambda_{\max}\left(\sum_k X_k\right)\ge t\right\}
\le d\exp\left(-\frac{t^2/2}{\sigma^2+Rt/3}\right),
$$

where

$$
\sigma^2:=\left\|\sum_k\mathbb E(X_k^2)\right\|.
$$

In particular, for every $t>0$, with probability at least $1-re^{-t}$,

 $$\lambda_{\max}\!\left(\sum_{j=1}^sX_j\right)
 \le\sqrt{2\sigma^2t}+\frac23Rt.$$

\end{theorem}

For a real matrix $F$, let
$s_1(F)\ge s_2(F)\ge\cdots$ denote its singular values, i.e., the square roots
of the eigenvalues of $F^{\mathsf T}F$ in nonincreasing order. Thus
$\norm{F}=s_1(F)$. We write $F[I,J]$ for the submatrix with row set $I$ and
column set $J$, and $J_{r,s}$ for the $r\times s$ all-ones matrix; when
$r=s$, we abbreviate this to $J_r$. For a real matrix $F$, we write $F^{\mathsf T}$ for its transpose.
For self-adjoint matrices $A$ and $B$, we write $A\preceq B$ if $B-A$ is positive semidefinite; equivalently, $x^{\mathsf T}Ax\le x^{\mathsf T}Bx$ for every vector $x$. In particular, if $0\preceq A\preceq B$, then $\norm{A}\le \norm{B}$. For a vector $f\in\mathbb R^r$, the matrix $ff^{\mathsf T}$ is the rank-one positive semidefinite matrix whose $(i,j)$-entry is $f_if_j$.

The proof of Lemma~\ref{lem:main} follows by applying the following lemma twice, once to sample the rows, and once to sample the columns.

\begin{lemma}\label{lem:columns}
Let $F\in\R^{r\times s}$, where $r,s\ge1$, with columns
$f_1,\ldots,f_s$, and set $c=\max_{j\in[s]}\norm{f_j}_2$.
If $I$ is a uniformly random $k$-subset of $[s]$, where $0\le k\le s$,
and $q=k/s$, let $F_I$ be the submatrix consisting of the columns indexed
by $I$. Then, for every $t>0$,
\begin{equation}\label{eq:column-sampling}
 \PP\!\left(\norm{F_I}>
       \sqrt q\,\norm{F}+c\sqrt t\right)
 \le r(s+1)e^{-t}.
\end{equation}
\end{lemma}

\begin{proof}
We assume that
$0<k<s$ and $c>0$, as otherwise the statement is trivial.
First select columns independently with probability $q$.
Write $f_j\in\R^r$ for column $j$ of $F$, and let $\xi_j$ be its indicator random variable. If $F_I$ denotes the selected matrix, then we observe that
\begin{equation}\label{eq:gram}
 F_IF_I^{\mathsf T}=\sum_{j=1}^s\xi_j f_jf_j^{\mathsf T}.
\end{equation}
Note also that  $\norm{F_I}^2=\lambda_{\max}(F_IF_I^{\mathsf T})$. To apply Matrix Bernstein, we need to center the matrices, so set $X_j=(\xi_j-q)f_jf_j^{\mathsf T}$, noting $\E X_j=0$. Since
$\norm{f_jf_j^{\mathsf T}}=\norm{f_j}_2^2$, we have
$\norm{X_j}\le c^2$. Moreover,
\begin{align*}
 \sum_j\E(X_j^2)
 &=q(1-q)\sum_j\norm{f_j}_2^2 f_jf_j^{\mathsf T}\\
 &\preceq qc^2\sum_j f_jf_j^{\mathsf T}
 =qc^2FF^{\mathsf T}.
\end{align*}
The variance parameter in Theorem~\ref{thm: matrix berstein} is therefore at most
$qc^2\norm{F}^2$. Theorem~\ref{thm: matrix berstein} implies that (using the `in particular' part), with failure probability at most $re^{-t}$, we have 
\begin{align*}
 \norm{F_I}^2
 &\le q\norm{F}^2+\sqrt{2q}\,c\norm{F}\sqrt t+\frac23c^2t\\
 &\le\bigl(\sqrt q\,\norm{F}+c\sqrt t\bigr)^2.
\end{align*}
This proves the desired bound under binomial sampling.

To pass to a uniform $k$-set, condition on $\sum_j\xi_j=k$.
For $q=k/s$, the integer $k$ is a mode of $\operatorname{Bin}(s,q)$,
so
\begin{equation}\label{eq:conditioning}
 \PP\!\left(\sum_j\xi_j=k\right)\ge\frac1{s+1}.
\end{equation}
Consequently, conditioning increases the failure probability by at most
$s+1$, which proves~\eqref{eq:column-sampling}.
\end{proof}

We can now start the proof of the subsampling lemma.

\begin{proof}[Proof of Lemma~\ref{lem:main}]
We denote $V(G)=[n]$ and set $t=(a+4)\log n$,  $\eta=\frac{\delta}{100}$. Our choice of $C_a$ gives $D\ge30000\delta^{-2}t$.

We first define a centered version of the matrix as $B=A-\frac dnJ_n$, noting $\norm{B}\le(1-\delta)d.$ 
Observe, by $d$-regularity, every row and column of $B$ has squared Euclidean norm
\begin{equation}\label{eq:row-norms}
 d\left(1-\frac dn\right)^2
 +(n-d)\left(\frac dn\right)^2
 =d-\frac{d^2}{n}\le d.
\end{equation}
Also $K=(d/n)J_m+B[U,V]$, so we can deduce
\begin{equation}\label{eq:rank-one}
 s_2(K)\le\norm{B[U,V]}.
\end{equation}
Above we used the inequality
$s_2(M)\le\norm{M-R}$ that holds when $\rank R\le1$, which follows from the min-max characterization of $s_2$. 

\smallskip
We first sample the rows. By transposing, we may apply Lemma~\ref{lem:columns} to the rows, and
with~\eqref{eq:row-norms}, we obtain
\begin{equation}\label{eq:sample-rows}
 \norm{B[U,[n]]}\le\sqrt p\,\norm{B}+\sqrt{dt}
\end{equation}
except with probability at most $n(n+1)e^{-t}$.
By Chernoff's bound and a union bound, we also have, with very high probability, for every $j\in[n]$,
\begin{align}
 \norm{B[U,\{j\}]}_2^2
 &=\sum_{i\in U}\left(A_{ij}-\frac dn\right)^2\notag\\
 &\le d_G(j,U)+m\left(\frac dn\right)^2
 \le2D+\frac dnD\le3D.
 \label{eq:restricted-columns}
\end{align}

\smallskip
We now condition on (\ref{eq:sample-rows}) and (\ref{eq:restricted-columns}), and sample the columns. The set $V$ is a uniform $m$-subset of $[n]\setminus U$,
so its proportion is $q=\frac{m}{n-m}=\frac p{1-p}$. Apply Lemma~\ref{lem:columns} to the matrix
$F=B[U,[n]\setminus U]$. With failure probability at most
$m(n-m+1)e^{-t}\le n(n+1)e^{-t}$, we have that
\begin{equation}\label{eq:sample-columns}
 \norm{B[U,V]}
 \le\sqrt q\,\norm{F}
   +\max_{j\notin U}\norm{B[U,\{j\}]}_2\sqrt t.
\end{equation}
Since this conditional failure bound is independent of the choice of $U$, it also holds unconditionally after averaging over $U$.

Deleting columns cannot increase the operator norm, so
$\norm{F}\le\norm{B[U,[n]]}$. Conditional on the high probability events,
\eqref{eq:sample-rows}, \eqref{eq:restricted-columns},
and~\eqref{eq:sample-columns}, we have that 
\begin{align}
 \norm{B[U,V]}
 &\le\sqrt{\frac p{1-p}}
       \bigl(\sqrt p\,\norm{B}+\sqrt{dt}\bigr)+\sqrt{3Dt}\notag\\
 &=\frac p{\sqrt{1-p}}\norm{B}
   +\left(\frac1{\sqrt{1-p}}+\sqrt3\right)\sqrt{Dt}\notag\\
 &\le\frac p{\sqrt{1-p}}\norm{B}+4\sqrt{Dt},
 \label{eq:compression-estimate}
\end{align}
where in the last inequality we used that $p\le1/2$.

We now combine everything.
By~\eqref{eq:rank-one} and using $D\geq 30000\delta^{-2}t$, we have $\frac{s_2(K)}D
 \le\frac{1-\delta}{\sqrt{1-p}}+4\sqrt{\frac tD}.
$
Using $p\le\delta/100$, $(1-p)^{-1/2}\le1+p$, and \eqref{eq:restricted-columns}, we obtain
$$
 \frac{s_2(K)}D
 \le1-\delta+p+\frac{\delta}{4}
 \le1-\frac{\delta}{2}.
$$
This proves the desired spectral bound~\eqref{eq:spectral-conclusion} conditional on the good events.
The sum of the failure probabilities in~\eqref{eq:restricted-columns}
and the two matrix sampling steps is at most $6n(n+1)e^{-t}
 =6n(n+1)n^{-(a+4)}
 \le n^{-a}$, as desired.
\end{proof}

\section{Making connections}

The following lemma allows us to connect a collection of pairs of vertices with vertex-disjoint paths in mildly expanding graphs. We will use it in the next section to build a Hamilton router.

\begin{lemma}[Connecting lemma]
\label{lem:cut-dense-connector}
There is an absolute constant $C>0$ such that the following holds. Let $G$ be a graph and let $R\subseteq V(G)$ with $|R|\ge 3$. Suppose that $G[R]$ is $p$-cut-dense, write $\Delta:=\Delta(G[R])$, and assume that $p|R|\ge \gamma\Delta
$ for some $0<\gamma\le1$. Let $(a_i,b_i)_{i=1}^k$ be pairs of vertices in $V(G)\setminus R$, and let $T:=\{a_i,b_i:i\in[k]\}$. Let $\beta\in (0,1)$ and suppose that every vertex occurs at most $10$ times among $a_1,b_1,\ldots,a_k,b_k$, that $
d_G(x,R)\ge \beta\Delta
\text{ for every }x\in T,
$ and that $
d_G(v,T)\le \frac{\beta\gamma^2\Delta}{C\log(2|R|)}
\text{ for every }v\in R.
$

Then, there exist $a_i$--$b_i$ paths $P_i$, $i\in[k]$, whose internal vertex sets are pairwise disjoint and contained in $R$, and such that $
|P_i|\le \frac{C}{\gamma}\log(2|R|)
\text{ for every }i\in[k].
$

\end{lemma}

One can compare the statement with
\cite[Definition~3.3 and Lemma~3.4]{CJMM}
and \cite[Lemma~5.1]{LetzterMethukuSudakov2026}, which obtains a similar conclusion with an additional property that the paths found are contained within a random reservoir.
In our lemma, the reservoir $R$ itself is assumed to expand well, hence
the reachability arguments in \cite{CJMM, LetzterMethukuSudakov2026} are not necessary. In our application, $R$ will still be a random subset, and we will show that $R$ inherits a cut-density condition from the expander it is sampled from, using the result in the previous section, Lemma~\ref{lem:main}. 

\par The proof thus reduces to a standard application of
Haxell's hypergraph matching theorem \cite{Haxell1995}, which we state below, as it was stated as Theorem 2 in \cite{AsadpourFeigeSaberi2008}. Recall that $\tau(\mathcal H)$ is the minimum number of vertices of $B$ needed to meet every edge of a hypergraph $\mathcal H$.

\begin{theorem}[Haxell] Let $\mathcal H$ be a bipartite hypergraph with vertex partition $A\cup B$, where every edge $e\in E(\mathcal H)$ satisfies
$|e\cap A|=1\text{ and }|e\cap B|\le r-1.$
For $S\subseteq A$, let $\mathcal H_S$ be the hypergraph on $B$ whose edge set is $E(\mathcal H_S)
:=
\{
e\cap B:
e\in E(\mathcal H),\ e\cap S\neq\varnothing
\}.$
If $\tau(\mathcal H_S)>
(2r-3)(|S|-1)\text{ for every }S\subseteq A,$ then $\mathcal H$ has a matching saturating $A$.
\end{theorem}

We essentially model our argument based on Section 4.1 in \cite{BucicHeHuangSaranurak2026}, the main difference is that \cite{BucicHeHuangSaranurak2026} is concerned with edge-disjointness, as opposed to vertex-disjointness. 

\begin{proof}[Proof of Lemma~\ref{lem:cut-dense-connector}] Set $\ell:=C_0\gamma^{-1}\log(2|R|)$ with $C_0$ a sufficiently large constant. Thus, $\ell$ denotes the target upper bound on the length of the connecting paths we wish to find. We will later choose $C\gg C_0$. We begin with showing that after deleting a few vertices, there still remains a suitable pair $a_i-b_i$ with a short connection, avoiding the deleted vertices. 

    \begin{claim} For every $I\subseteq [k]$ and $U\subseteq R$ with $|U|\leq (2\ell-3)(|I|-1)$, there is some $i\in I$ and a $a_i-b_i$ path of length at most $\ell$ whose internal vertices are in $R\setminus U$.
\end{claim}

We first show how to conclude the proof. Define an auxiliary bipartite hypergraph $\mathcal{H}$ on vertex-partition $A\cup B$, where $A$ corresponds to the $k$ pairs $a_i-b_i$, and $B$ corresponds to vertices of $R$, and put an edge $a\cup b$ where $a\in A$ and $b\subseteq B$ if the vertices of $b$ form the internal vertices of a $a_i-b_i$ path of length at most $\ell$. Note that a matching saturating $A$ is exactly what we are looking for. Furthermore, the assumption of the claim translates exactly to the vertex-cover hypothesis of Haxell's theorem, as desired. It remains to check the claim.

\begin{proof}[Proof of claim] Write $n:=|R|$ and $H:=G[R]$. We assume that $C_0$ is a sufficiently large absolute constant and that $C$ is sufficiently large compared with $C_0$. Recall that $\ell=C_0\gamma^{-1}\log(2n)$.

We first give an upper bound on the size of $I$. Using that every vertex occurs at most $10$ times among the endpoints of the pairs, we have
$
2\beta\Delta |I|
\le \sum_{i\in I}\big(d_G(a_i,R)+d_G(b_i,R)\big)
\le 10\sum_{v\in R}d_G(v,T)
\le \frac{10\beta\gamma^2\Delta n}{C\log(2n)}.
$
Hence
$
|I|\le \frac{5\gamma^2 n}{C\log(2n)}.
$
It follows from the assumption on $U$ that
$
|U|\le 2\ell |I|\le \frac{10C_0}{C}\gamma n.
$
In particular, taking $C$ sufficiently large compared with $C_0$, we may assume that $|U|\le\gamma n/100$.

We next find a large expanding subgraph of $H-U$. Let $W\subseteq R\setminus U$ be a set of maximum size, subject to $|W|\le 3n/4$ and
$
e_H\big(W,R\setminus(U\cup W)\big)<\frac{\gamma\Delta}{16}|W|,
$
taking $W=\varnothing$ if there is no non-empty set with this property. We claim that
$
|W|\le \frac{6|U|}{\gamma}.
$
Indeed, if $W\neq\varnothing$, then cut-density, $|W|\le3n/4$, and $pn\ge\gamma\Delta$ give
$
e_H(W,R\setminus W)\ge p|W|(n-|W|)\ge\frac{\gamma\Delta}{4}|W|.
$
Consequently,
$
e_H(W,U)>\frac{3\gamma\Delta}{16}|W|.
$
Since $e_H(W,U)\le\Delta|U|$, the claimed bound follows.

Set $K:=H-(U\cup W)$. We claim that every $S\subseteq V(K)$ with $|S|\le |V(K)|/2$ satisfies
$
e_K(S,V(K)\setminus S)\ge\frac{\gamma\Delta}{16}|S|.
$
Suppose otherwise, and put $W':=W\cup S$. By the bound on $W$ and the fact that $|U|\le\gamma n/100$, we have
$
|W'|\le |W|+\frac{n-|U|-|W|}{2}<\frac{3n}{4}.
$
Moreover,
$
e_H\big(W',R\setminus(U\cup W')\big)
\le e_H\big(W,R\setminus(U\cup W)\big)
+e_K(S,V(K)\setminus S)
<\frac{\gamma\Delta}{16}|W'|,
$
contradicting the maximality of $W$.

As $\Delta(K)\le\Delta$, it follows that
$
|N_K(S)\setminus S|\ge\frac{\gamma}{16}|S|
$
whenever $|S|\le |V(K)|/2$. Thus $K$ is a $\gamma/16$-vertex-expander. By a standard ball growth argument, $K$ has diameter at most $64\gamma^{-1}\log(2n)$. Taking $C_0$ sufficiently large, we may therefore assume that
$
\mathrm{diam}(K)+2\le\ell.
$

It remains to show that some pair has both endpoints adjacent to $K$. Suppose not. For every $i\in I$, choose $x_i\in{a_i,b_i}$ with $N_G(x_i,R)\subseteq U\cup W$. Since each vertex occurs at most $10$ times among the endpoints,
$
\beta\Delta |I|
\le \sum_{i\in I}d_G(x_i,U\cup W)
\le 10\sum_{v\in U\cup W}d_G(v,T).
$
Using $|W|\le6|U|/\gamma$, the hypothesis on $d_G(v,T)$, and $\gamma\le1$, the right hand side is at most
$
10\left(|U|+\frac{6|U|}{\gamma}\right)
\frac{\beta\gamma^2\Delta}{C\log(2n)}
\le \frac{70\beta\gamma\Delta |U|}{C\log(2n)}
\le \frac{140C_0}{C}\beta\Delta |I|.
$
This is a contradiction once $C$ is chosen sufficiently large compared with $C_0$.

Hence, for some $i\in I$, both $a_i$ and $b_i$ have a neighbour in $K$. Let $x\in N_G(a_i)\cap V(K)$ and $y\in N_G(b_i)\cap V(K)$. A shortest $x$--$y$ path in $K$, together with the edges $a_ix$ and $yb_i$, gives an $a_i$--$b_i$ path whose internal vertices lie in $R\setminus U$ and whose length is at most $\mathrm{diam}(K)+2\le\ell$, as required.
\end{proof}
This concludes the proof.\end{proof}

\section{Hamilton routers}
We require the definition of a Hamilton router and an explicit construction as given in Definition~\ref{defn:basic hamilton routers}. Both are borrowed entirely from our previous work with Bowtell, Montgomery, and Pokrovskiy \cite{bowtell2026proof}. Sorting networks, as previously used in the work of Hyde et al. \cite{hyde2025spanning}, could be used for the construction as well. However, this latter construction is more complicated than we need for our purposes here, and costs additional logarithmic factors.

\begin{defn}[Hamilton router] Suppose $A$ and $B$ are disjoint subsets of $V(S)$ with $|A|=|B|$ for some graph $S$. We call $S$ an $A,B$\textit{-Hamilton-router} if, for every bijection $\phi\,\colon A\to B$, $V(S)$ can be partitioned into a collection of paths, $\mathcal{P}_\phi$, with endpoints in $A\times B$ so that, letting $M_\phi$ be the matching $\{(\phi(a),a)\colon a\in A\}\subseteq B\times A$, the edge-set $E(\mathcal{P}_\phi)\cup M_\phi$ forms a Hamilton cycle over $V(S)$. The \textit{order} of an $A,B$-Hamilton-router is $|A|=|B|$ and the \emph{depth} of an $A,B$-Hamilton-router is $|V(S)|/|A|-1$ (i.e., the average length of a component path in any path partition $\mathcal{P}_\phi$).
\end{defn}

\noindent \textbf{Remark.} Throughout the paper, we will refer to Hamilton routers of order 2 as \emph{comparators}. \vspace{2mm}  

\par The following proposition, whose proof is immediate by definition, explains why Hamilton routers are useful.
\begin{prop}\label{prop:basic}
    Let $D$ be a directed graph containing an $A,B$-Hamilton-router $S$. Suppose that there exists a directed path forest $\mathcal{P}$ with exactly $|A|=|B|$ paths, start-vertices in $B$, end-vertices in $A$, and internal vertices partitioning $V(D)\setminus V(S)$. Then, $D$ has a directed Hamilton cycle.
\end{prop}

\subsection{Constructing Hamilton routers}\label{sec:routers}

We can build Hamilton routers of order $n$ by gluing together Hamilton routers of order $2$, i.e., comparators (as depicted in Figure~\ref{fig:activated}), as was done by Bowtell et al. \cite{bowtell2026proof}.

\begin{defn}\label{defn:basic hamilton routers}
    We call an $A,B$-Hamilton-router $S$ a \textit{basic Hamilton router of order $n$} if $S$ can be constructed as follows. Start with a vertex set $V$ defined on a $4\times n$ grid, that is, $V:=\{v_{i,j}\,\colon (i,j)\in [4]\times [n]\}$. Let $A=\{v_{1,j}\,\colon j\in  [n]\}$ and $B=\{v_{4,j}\,\colon j\in  [n]\}$.
    \begin{itemize}
        \item For every odd $i\in [n-1]$, and for $A'=\{v_{1,i},v_{1,i+1}\}$ and $B'=\{v_{2,i},v_{2,i+1}\}$, add a graph $S^{\textrm{odd}}_i$ which is an $A',B'$-comparator.
        \item For every even $i\in [n-1]$, and for $A'=\{v_{3,i},v_{3,i+1}\}$ and $B'=\{v_{4,i},v_{4,i+1}\}$, add a graph $S^{\textrm{even}}_i$ which is an $A',B'$-comparator.
        \item For every $i\in [n]$, add a path from $v_{2,i}$ to $v_{3,i}$.
        \item Add an edge from $v_{3,1}$ to $v_{4,1}$. If $n$ is even, add an edge from $v_{3,n}$ to $v_{4,n}$, and otherwise add an edge from $v_{1,n}$ to $v_{2,n}$.
    \end{itemize}
\end{defn}

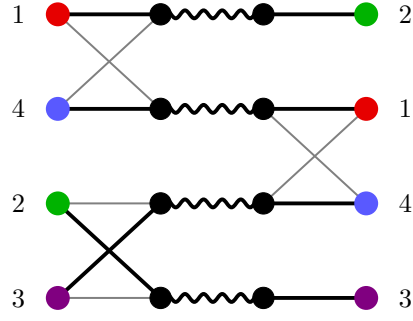
\begin{figure}[h]
    \centering

\begin{tikzpicture}[
    x=1.7cm,
    y=1.25cm,
    straight/.style={line width=1.45pt, line cap=round},
    thinline/.style={line width=0.75pt, line cap=round,black!50},
    wavy/.style={
        line width=1.45pt,
        decorate,
        decoration={snake, amplitude=1.5pt, segment length=7pt},
        line cap=round
    },
    blackvertex/.style={circle, fill=black, draw=black, inner sep=0pt, minimum size=8pt},
    redvertex/.style={circle, fill=red!90!black, draw=none, inner sep=0pt, minimum size=9pt},
    bluevertex/.style={circle, fill=blue!65, draw=none, inner sep=0pt, minimum size=9pt},
    greenvertex/.style={circle, fill=green!70!black, draw=none, inner sep=0pt, minimum size=9pt},
    purplevertex/.style={circle, fill=violet, draw=none, inner sep=0pt, minimum size=9pt}
]

\coordinate (L1) at (0,3);
\coordinate (L2) at (0,2);
\coordinate (L3) at (0,1);
\coordinate (L4) at (0,0);

\coordinate (A1) at (0.8,3);
\coordinate (A2) at (0.8,2);
\coordinate (A3) at (0.8,1);
\coordinate (A4) at (0.8,0);

\coordinate (B1) at (1.6,3);
\coordinate (B2) at (1.6,2);
\coordinate (B3) at (1.6,1);
\coordinate (B4) at (1.6,0);

\coordinate (R1) at (2.4,3);
\coordinate (R2) at (2.4,2);
\coordinate (R3) at (2.4,1);
\coordinate (R4) at (2.4,0);

\draw[straight] (L1) -- (A1);
\draw[straight] (L2) -- (A2);
\draw[thinline] (L1) -- (A2); 
\draw[thinline] (L2) -- (A1); 

\draw[thinline] (L3) -- (A3); 
\draw[thinline] (L4) -- (A4); 
\draw[straight] (L3) -- (A4);
\draw[straight] (L4) -- (A3);

\draw[wavy] (A1) -- (B1);
\draw[wavy] (A2) -- (B2);
\draw[wavy] (A3) -- (B3);
\draw[wavy] (A4) -- (B4);

\draw[straight] (B1) -- (R1);
\draw[straight] (B4) -- (R4);

\draw[straight] (B2) -- (R2); 
\draw[straight] (B3) -- (R3); 
\draw[thinline] (B2) -- (R3); 
\draw[thinline] (B3) -- (R2); 

\node[redvertex]    at (L1) {};
\node[bluevertex]   at (L2) {};
\node[greenvertex]  at (L3) {};
\node[purplevertex] at (L4) {};

\node[greenvertex]  at (R1) {};
\node[redvertex]    at (R2) {};
\node[bluevertex]   at (R3) {};
\node[purplevertex] at (R4) {};

\node[blackvertex] at (A1) {};
\node[blackvertex] at (A2) {};
\node[blackvertex] at (A3) {};
\node[blackvertex] at (A4) {};

\node[blackvertex] at (B1) {};
\node[blackvertex] at (B2) {};
\node[blackvertex] at (B3) {};
\node[blackvertex] at (B4) {};

\node[anchor=east] at (-0.18,3) {$1$};
\node[anchor=east] at (-0.18,2) {$4$};
\node[anchor=east] at (-0.18,1) {$2$};
\node[anchor=east] at (-0.18,0) {$3$};

\node[anchor=west] at (2.58,3) {$2$};
\node[anchor=west] at (2.58,2) {$1$};
\node[anchor=west] at (2.58,1) {$4$};
\node[anchor=west] at (2.58,0) {$3$};

\end{tikzpicture}
    \caption{The above figure and following caption is borrowed from \cite{bowtell2026proof}. A basic Hamilton router of order $4$ where (for simplicity) the comparators are depicted as $4$-cycles. $A$ and $B$ are the leftmost and rightmost columns of vertices, respectively. An example bijection between $A$ and $B$ is indicated by the colour/number of the vertices. Choices of path systems for each comparator are indicated with thick lines, chosen so that identifying the vertex pairs of the same colour/number forms a cycle in the thick lines. }
    \label{fig:activated}
\end{figure}

The short argument justifying why the above construction is a valid Hamilton router is given in \cite{bowtell2026proof}.
\begin{prop}[\cite{bowtell2026proof}]\label{prop:routers are good} Basic Hamilton routers are well defined, that is, every (directed) graph constructed as in Definition~\ref{defn:basic hamilton routers} is an $A,B$-Hamilton-router.
\end{prop}

\subsection{Building comparators out of even cycles}

We want to embed Hamilton routers in mild expanders, so in particular, we need to be able to find comparators in such expanders. Expanders can have arbitrarily large girth, so $C_4$ is not a valid candidate for a comparator. Below we give a construction of comparators that can be built out of a single even cycle, followed by adding internally vertex-disjoint paths connecting various vertices of this base cycle. This eventually lets us find comparators with about $\log(n)^2$ vertices in expanders (after passing through Lemma~\ref{lem:cut-dense-connector}). The idea is basically borrowed from \cite{hyde2025spanning}, with the difference being that there is no modular restriction on the length of the even cycle. The construction we give here is also arguably simpler (but does not have the property of keeping the path lengths between the vertices the same, which is useful whilst embedding spanning trees).

\begin{lemma}\label{lemma: comparators exist}
    For every even $k\geq 4$, there exists a graph $G_k$ which is an $A,B$-Hamilton-router, where $A,B\subseteq V(G_k)$ are disjoint sets with $|A|=|B|=2$, i.e. $G_k$ is a comparator (recall Remark 1), and furthermore, $G_k$ has the following properties.
    \begin{enumerate}
        \item $\Delta(G_k)\leq 3$
        \item $G_k$ can be obtained by starting with a $k$-cycle $C$, identifying some $A,B\subseteq V(C)$, appropriately pairing up the vertices of $V(C)\setminus (A\cup B)$, and, for each pair, adding a path of arbitrary length between its two vertices, so that the added paths are pairwise vertex-disjoint and each meets $C$ precisely in its two endpoints.
    \end{enumerate}
\end{lemma}
\begin{proof}
    Property $2$ in the lemma statement essentially gives away the construction of the comparator, save for how to identify $A$ and $B$ along the base cycle, and how to pair up the vertices of $V(C)\setminus (A\cup B)$. For both, there are plenty of choices that work, here we give an arbitrary one. 
    \par Let $C$ be a $k$-cycle. If $k=4$, there is nothing to do, as this already forms a comparator. Otherwise, enumerate the vertices as $c_0c_1c_2c_3c_4\cdots c_{k-1}c_0$. Let $A=\{c_0,c_2\}$ and $B=\{c_1, c_4\}$. Let $M_0$ be the perfect matching of $C$ matching $c_0$ to $c_2$, and let $M_1$ be the other perfect matching (which pairs $c_1$ with $c_2$). 

    \begin{claim}
        There exists a perfect matching $P$ of $V(C)\setminus (A\cup B)$ so that $M_0\cup P$ and $M_1\cup P$ are both acyclic.
    \end{claim}

    Assuming the claim, $P$ gives us the necessary pairing, and we argue that any $G_k$ obtained by subdividing the edges of $P$ (an arbitrary number of times) forms a valid $A,B$-Hamilton-router. Indeed,  $M_0\cup P$ and $M_1\cup P$ are both spanning subgraphs where each vertex has degree $1$ or $2$ so both subgraphs are a union of cycles and paths, and there cannot be cycles by the claim, so just paths. The vertices in $A\cup B$ are the only vertices of degree $1$, so these are the endpoints of the paths. As $M_0\cup P$ has a path connecting $c_0$ and $c_1$, and $M_1\cup P$ has a path connecting $c_1$ and $c_2$, it is clear that in each case, the remaining path is between the correct pair of vertices ($c_2-c_4$ and $c_0-c_4$, respectively), verifying the definition of a $A,B$-Hamilton-router.
    \par The claim itself can be verified by giving an explicit pairing. For example, $$P=\{c_{2j+1}c_{2j+4}:1\le j\le r-3\}\cup\{c_{2r-3}c_{2r-1}\}$$
    is a pairing that does not create any cycles when combined with $M_0$ or $M_1$, see Figure~\ref{fig:C10}. 
\end{proof}
\begin{figure}[h]
    \centering
    \includegraphics[width=0.34\linewidth]{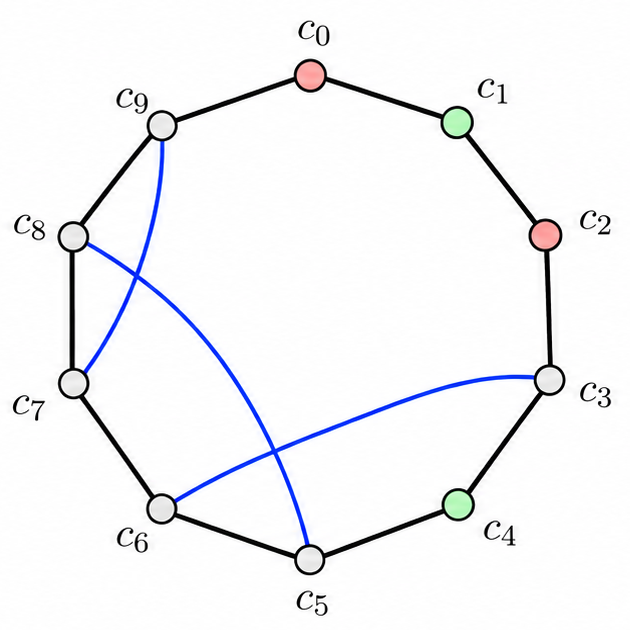}
    \caption{The graph $G_k$ constructed in Lemma~\ref{lemma: comparators exist} when $k=10$. The blue edges indicate the pairing given by $P$. Any graph in which the blue edges are subdivided, i.e. replaced by paths of arbitrary length with the same endpoints,  forms a valid comparator.}
    \label{fig:C10}
\end{figure}

\subsection{Finding even cycles}
To use the previous comparator construction, we need to be able to find many even length cycles in a given graph. We do not need expansion for this, just a degree bound. In particular, the following lemma can be proved easily by iterating the simple Moore bound, bounding the number of vertices of a graph in terms of its average degree and its girth, as given by a breadth-first search argument, see \cite[Theorem~1]{AlonHooryLinial}. 
\begin{lemma}[Packing short even cycles]\label{lem:cycles} Let $G$ be an $n$-vertex graph with $D\leq \delta(G)$ and $\Delta(G)\leq 2D$ for some $D\geq 100$. Then, $G$ contains at least $n/(100\log n)$ vertex-disjoint even-length cycles each of length at most $3\log n$. 
\end{lemma}
\begin{proof}
Let $H$ be a spanning bipartite subgraph of $G$ obtained from a maximum cut. Then $e(H)\ge e(G)/2\ge Dn/4$ and $\Delta(H)\le2D$. Let $\mathcal C$ be a maximal collection of vertex-disjoint cycles in $H$, each of length at most $3\log n$, and suppose that $|\mathcal C|<n/(100\log n)$. If $X$ is the union of these cycles, then $|X|<3n/100$, and hence
$$
e(H-X)\ge e(H)-2D|X|>\frac{Dn}{4}-\frac{6Dn}{100}=\frac{19Dn}{100}.
$$
Thus $H-X$ has average degree greater than $19D/50\ge38$. By the Moore bound for irregular graphs \cite[Theorem~1]{AlonHooryLinial}, an $N$-vertex graph of average degree at least $38$ contains a cycle of length at most
$
2\left\lceil\frac{\log(N/2)}{\log 37}\right\rceil+2<3\log n.
$
Since $H-X$ is bipartite, this cycle is even, contradicting the maximality of $\mathcal C$. Hence $|\mathcal C|\ge n/(100\log n)$.
\end{proof}

\subsection{Finding Hamilton routers in mild expanders}

We can now combine our tools to give an embedding of a Hamilton router inside a random reservoir within an expander graph. 
\begin{lemma}[Random subsets of expanders contain Hamilton routers]
\label{lem:random-router}
There exist absolute constants $c>0$, $K>1$ and $n_0$ such that the following holds. Let $H$ be an $(n,d,\lambda)$-graph with $n\ge n_0$ and $\lambda\le (1-\delta)d$ for some $0<\delta\le 1/2$. Let $0<q\le \delta/100$ with $qn\in 2\mathbb{N}$, and suppose that $qd\ge K\delta^{-4}(\log n)^2.$ Let $R\subseteq V(H)$ be chosen among the $qn$-element subsets of $V(H)$ uniformly at random. Then, with probability at least $1-1/n$, the following holds simultaneously for all disjoint sets $A,B\subseteq V(H)\setminus R$ with $|A|=|B|=k$.

If every $x\in A\cup B$ satisfies $d_H(x,R)\ge \frac{qd}{2},$ every $v\in R$ satisfies $d_H(v,A\cup B)\le c\frac{\delta^3qd}{\log n},$ and it holds that $k\le c\delta^3|R|(\log n)^{-2}$, then $H[A\cup B\cup R]$ contains an $A,B$-Hamilton router $S$ such that $A\cup B\subseteq V(S)\subseteq A\cup B\cup R.$
\end{lemma}
\begin{proof}
 Let $c_0>0$ be a sufficiently small absolute constant, set $s:=\lfloor c_0\delta^2qn/\log n\rfloor$, and set $m:=(qn-s)/2$. We may generate a uniform $qn$-set $R$ by first choosing disjoint sets $C,U,V$ uniformly with $|C|=s$ and $|U|=|V|=m$, and then setting $R=C\cup U\cup V$. Write $D_0:=md/n$ and $D_C:=sd/n$. For $K$ sufficiently large, $D_0\ge qd/3$ and $D_C\ge c_0\delta^2qd/(2\log n)\ge C\log n$.

Apply Lemma~\ref{lem:main} to $(U,V)$ with $a=4$. Standard hypergeometric Chernoff bounds, followed by a union bound over all vertices, also give simultaneously $d_H(v,U),d_H(v,V)=(1\pm\delta/100)D_0$ and $d_H(v,C)=(1\pm1/10)D_C$ for every $v\in V(H)$. Thus, with probability $1-o(n^{-1})$,

$$
s_2(H[U,V])\le(1-\delta/2)D_0,\qquad
d_{H[U,V]}(x)=(1\pm\delta/100)D_0,
\qquad
d_H(v,C)\le 2D_C
$$

for all relevant vertices. We henceforth fix such a partition.

Let $J:=H[U,V]$, viewed as a bipartite graph on $Q$. Since the second largest adjacency eigenvalue of $J$ is $s_2(H[U,V])$, Lemma~\ref{lemma: spectral cut density} gives that $J$ is $\rho$-cut-dense for $\rho:=\delta d/(5n)$. Moreover, $\Delta(J)\le(1+\delta/100)D_0$, and therefore $\rho|Q|\ge \frac{\delta}{3}\Delta(J).$ In particular, Lemma~\ref{lem:cut-dense-connector} applies to $J$ with $\gamma=\delta/3$ provided the degree conditions hold.

Next, we embed the base cycles of the comparators. The graph $H[C]$ has minimum degree at least $D_C/2$ and maximum degree at most $2D_C$. By Lemma~\ref{lem:cycles}, $H[C]$ contains at least $c_1|C|\log(n)^{-1}$ vertex-disjoint even cycles, each of length at most $g:=C_1\log n$. Since
$k\le c\delta^3qn\log^{-2}n$ and $|C|=\Theta(\delta^2qn/\log n)$, choosing $c$ sufficiently small ensures that there are at least $k$ such cycles. A basic Hamilton router of order $k$ uses fewer than $k$ comparators, so we can assign a distinct cycle to each comparator and apply Lemma~\ref{lemma: comparators exist} to designate its four input/output vertices $(A,B)$ and pair its remaining vertices.

It remains to form all the required connections between the cycles, both to complete the even cycles into comparators and to join the comparators following the template of a basic Hamilton router. Form a collection $\mathcal D$ of pairs whose endpoints lie in $A\cup B\cup C$ as follows. Include all pairs used in the comparator constructions, and include one pair for every fixed connection joining the inputs/outputs of the comparators in the definition of a basic Hamilton router. Every vertex occurs only a bounded number of times in $\mathcal D$ (in fact the construction gives much less than $10$). Every endpoint $x$ of a pair in $\mathcal D$ has at least $qd/3$ neighbours in $Q$. Indeed, for $x\in A\cup B$ this follows from $d_H(x,R)\ge qd/2$ and $d_H(x,C)\le2D_C$, and for $x\in C$ it follows directly from the degree concentration into $U$ and $V$.

Let $T$ be the set of vertices occurring in $\mathcal D$. Since all cycle inputs/outputs lie in $C$, for every $v\in Q$ we have
$d_H(v,T)\le d_H(v,C)+d_H(v,A\cup B)\le 2D_C+c\delta^3qd/\log n$. By the choice of $s$, this is at most $c_2\delta^2qd/\log n$, where $c_2$ can be made arbitrarily small by choosing $c_0$ and then $c$ sufficiently small. As $\Delta(J)=\Theta(qd)$, this is precisely the condition required in Lemma~\ref{lem:cut-dense-connector} with $\beta$ an absolute constant and $\gamma=\Theta(\delta)$. The same lemma therefore provides internally vertex-disjoint paths through $Q$ joining every pair in $\mathcal D$.

Combining these paths with the base cycles gives the required comparators by Lemma~\ref{lemma: comparators exist}, and the remaining connecting paths give the fixed connections of a basic Hamilton router. Replacing those connections by subdivisions does not affect the proof of Proposition~\ref{prop:routers are good}, so their union is an $A,B$-Hamilton router $S$ with $A\cup B\subseteq V(S)\subseteq A\cup B\cup R$.
\end{proof}

\section{Checking Hall's condition between random sets after perturbation}






Below we give a convenient criterion for when a balanced bipartite graph contains a perfect matching. In the below, $H_0$ is a balanced bipartite graph, which, in our application, will be obtained by random sampling from a regular expander. We will show such an $H_0$ is nearly-regular and has cut-density (using Chernoff and Lemma~\ref{lem:main}), as assumed in the below statement. The cut-density parameter exceeds the error from near-regularity, so Hall's condition holds for $H_0$ with some slack. This slack is sufficient to say that $H_0$ still has a perfect matching even if a few more vertices are added each of which have large degree towards $H_0$, provided that each vertex of $H_0$ sends not too many edges towards the new vertices. 
\par The proof is a simple and standard counting argument, we give the details for the sake of completeness.

\begin{lemma}
\label{lem:perturbed-perfect-matching}
Let $H_0$ be a balanced bipartite graph with bipartition $(L_0,R_0)$, where $|L_0|=|R_0|=m$. Suppose that $H_0$ is $p$-cut-dense and
$d_{H_0}(v)=(1\pm\eta)D$ for every $v\in V(H_0)$. Set $\alpha:=pm/D$, and suppose that $0<\alpha\le 1/2$ and $\eta\le\alpha/10$.

Let $X,Y$ be disjoint sets with $|X|=|Y|$, and let $H$ be a bipartite graph with bipartition $(L_0\cup X,R_0\cup Y)$ containing $H_0$. Suppose that every $x\in X$ has at least $D/2$ neighbours in $R_0$, every $y\in Y$ has at least $D/2$ neighbours in $L_0$, and

$$
d_H(v,X)\le\frac{\alpha D}{20}\quad\text{for every }v\in R_0,
\qquad
d_H(v,Y)\le\frac{\alpha D}{20}\quad\text{for every }v\in L_0.
$$

Then $H$ contains a perfect matching.
\end{lemma}

\begin{proof}
If $S\subseteq L_0$ and $T\subseteq R_0$ satisfy $e_{H_0}(S,T)=0$, then

$$
m-|S|-|T|\ge\alpha\min\{|S|,|T|\}.
$$

Indeed, write $s:=|S|$, $t:=|T|$, $h:=m-s-t$, and assume $s\le t$. For $Z:=S\cup(R_0\setminus T)$ we have $|Z|=2s+h\le m$. Hence cut-density and near-regularity give
$p m(2s+h)\le e_{H_0}(Z,V(H_0)\setminus Z)\le D(2\eta s+(1+\eta)h)$. Since $pm=\alpha D$ and $\eta\le\alpha/10$, this implies $h\ge\alpha s$, proving the claim.

Suppose now that $H$ has no perfect matching. By Hall's theorem there are sets $S'\subseteq L_0\cup X$ and $T'\subseteq R_0\cup Y$ with $e_H(S',T')=0$ and $|S'|+|T'|>m+|X|$. Write $S=S'\cap L_0$, $T=T'\cap R_0$, $x:=|S'\cap X|$, $y:=|T'\cap Y|$, and $h:=m-|S|-|T|$. Since $x,y\le |X|$, Hall's condition failing gives $h<\min\{x,y\}$.

Every vertex of $S'\cap X$ sends at least $D/2$ edges into $R_0\setminus T$, while every vertex there receives at most $\alpha D/20$ such edges. Thus $x\le(\alpha/10)(m-|T|)$, and similarly $y\le(\alpha/10)(m-|S|)$. Consequently, writing $u:=\min\{|S|,|T|\}$,
$h<\frac{\alpha}{10}(h+u)$. On the other hand, the claim above gives $h\ge\alpha u$. Hence
$h<\frac{\alpha+1}{10}h\le\frac{3}{20}h$, a contradiction.
\end{proof}

\section{Putting everything together}

\begin{proof}[Proof of Theorem~\ref{thm: main theorem}]
We first describe the proof at a high level. We first choose a small random set $R$, and partition the remainder randomly into equal-sized sets. We will sometimes call $R$ a \textit{reservoir} and the latter random sets \textit{layers}. We build a Hamilton router with terminal sets $A,B$ inside $A\cup B\cup R$, where $A$ and $B$ are layers, then redistribute every vertex left outside the router among the remaining layers and use Lemma~\ref{lem:perturbed-perfect-matching} to find perfect matchings between consecutive layers. The resulting $k$ disjoint $B$--$A$ paths from the union of perfect matchings cover everything outside the router, so Proposition~\ref{prop:basic} finishes the proof.

We suppress divisibility issues throughout to ease exposition. In particular, we assume that all set sizes below are integers and that $k$ divides $n-|V(S)|$ once the router $S$ has been found. These assumptions can be removed by starting the proof by finding a short path whose remainder is appropriately divisible, contracting the vertices of the short path, and essentially following the below proof verbatim. 

Choose sufficiently small absolute constants $0<c_1\ll c_0\ll1$, and put $q:=c_0\delta$ and $k:=c_1\delta^3qn/L^2$. We may assume $0<\delta\le1/2$, since the case of larger $\delta$ is only easier. Set $h:=2qk$ and $m:=k-h=(1-2q)k$, and let
$T:=\lfloor(n-qn-2k)/m\rfloor$. Note that $T=\Theta(n/k)$ and, since $qn\gg k$, we have $T\ge n/k$ for $n$ sufficiently large.

Choose uniformly at random a partition

$$
V(G)=R\sqcup A_0\sqcup A_1\sqcup B_0\sqcup B_1\sqcup C_1\sqcup\cdots\sqcup C_T\sqcup Z,
$$

where $|R|=qn$, $|A_0|=|B_0|=|C_i|=m$ and $|A_1|=|B_1|=h$. Set $A:=A_0\cup A_1$ and $B:=B_0\cup B_1$, so $|A|=|B|=k$.

\textbf{Part 1.} We first record several properties of this partition that each hold with high probability.

\par We start with the spectral properties of the various induced bipartite subgraphs. Set $D:=md/n$. Apply Lemma~\ref{lem:main} (with $p\to m/n$ and the same value of $D,\delta$), with a sufficiently large fixed value of $a$, to every pair $(C_i,C_{i+1})$, to $(B_0,C_1)$, and to $(C_i,A_0)$ for every $i\le T$. Each such pair is genuinely a uniformly random ordered pair of disjoint $m$-sets of the original graph, and the hypotheses $p\le\delta/100$ and $D\ge C_a\delta^{-2}\log n$ hold because $m=(1-2q)k$, $q=c_0\delta$, and $k=c_1\delta^3qn/(\log n)^2$, so, for $c_0,c_1$ sufficiently small and $n$ sufficiently large,
$$
\frac{m}{n}\le \frac{k}{n}=c_0c_1\frac{\delta^4}{(\log n)^2}\le \frac{\delta}{100},
\qquad
D\ge \frac{c_0c_1}{2}\frac{\delta^4d}{(\log n)^2}\ge C_a\delta^{-2}\log n,
$$
where the last inequality follows from $d\ge C\delta^{-6}(\log n)^3$ once the absolute constant $C$ is chosen sufficiently large. 

The conclusion of Lemma~\ref{lem:main} ensures that in each application, the corresponding random layers form a bipartite graph with good spectral gap, i.e. the corresponding biadjacency matrix $K$ has $s_2(K)\leq (1-\delta/2)D$ (meaning the adjacency matrix $A$ of the bipartite graph has second largest eigenvalue at most $(1-\delta/2)D$), with probability $1-n^{-a}$. As $T=O(\delta^{-4}\log(n)^2)$, we may take a union bound to conclude with high probability, this conclusion holds for all pairs of random sets.
\par We now record the concentration of degrees. Set $\eta:=\delta/100$. Standard hypergeometric concentration gives that with high probability 
$d_G(v,X)=(1\pm\eta)D$ holds for every vertex $v$ and every layer $X\in{A_0,B_0,C_1,\ldots,C_T}$. We can similarly assert, with high probability, the corresponding degree concentration of each vertex into $\{R,A_0,A_1,B_0,B_1\}$ and any other random set that is obtained by unions and complements from this set, such as $A$ and $B$.

Consequently, for every relevant pair of random layers, the balanced bipartite graph between them has all degrees $(1\pm\eta)D$ and second adjacency eigenvalue at most $(1-\delta/2)D$. Lemma~\ref{lemma: spectral cut density} therefore shows that each such pair is $p_*$-cut-dense, where we may safely take $p_*:=\frac{\delta d}{10n}$.

Finally, by Lemma~\ref{lem:random-router}, with probability at least $1-1/n$ the random set $R$ has the required Hamilton router property simultaneously for all admissible $A,B$. 

\textbf{Part 2.} By the probabilistic method, we may fix a partition that satisfies all of the properties above. The rest of the proof uses these properties deterministically to conclude the existence of a Hamilton cycle.

Let us start by considering the reservoir $R$, in which we wish to embed a Hamilton router. By the degree bounds, the sets $A$ and $B$ satisfy the hypotheses from Lemma~\ref{lem:random-router}. Indeed, every $x\in A\cup B$ has $d_G(x,R)\ge qd/2$, while every $v\in R$ has
$d_G(v,A\cup B)\le 4kd/n\le c\delta^3qd/L$, provided $c_1$ was chosen sufficiently small. Since $k\le c\delta^3|R|\log\log n/\log(n)^2$, the property from Lemma~\ref{lem:random-router} gives an $A,B$-Hamilton router
$S\subseteq G[A\cup B\cup R]$.

\par Next, we wish to partition the vertices outside the router $S$, and combine them with the layers $C_1,\ldots, C_T$, to form various balanced bipartite graphs between which we intend to find perfect matchings. We will end up discarding a few layers $C_i$ for convenience, and redistribute. Set $t:=(n-|V(S)|)/k$. Since $S\subseteq A\cup B\cup R$, we have $t\le n/k\le T$. Let
$P:=V(G)\setminus\bigl(V(S)\cup C_1\cup\cdots\cup C_t\bigr)$. By the divisibility hypothesis and the definition of $m$, $|P|=t(k-m)=th$. Partition $P$ uniformly at random into sets $X_1,\ldots,X_t$ of size $h$, and set $L_i:=C_i\cup X_i$. Thus every $L_i$ has size $k$ and
$V(G)\setminus V(S)=L_1\sqcup\cdots\sqcup L_t$.

\par The $X_i$s are random sets at the moment, and we wish to show that with high probability, they are `admissible' with respect to the deterministic sets $C_i$ from the perspective of the hypotheses of Lemma~\ref{lem:perturbed-perfect-matching}. Set $\alpha:=\frac{p_*m}{D}=\frac{\delta}{10}$, noting $\eta\le\alpha/10$, exactly as required in Lemma~\ref{lem:perturbed-perfect-matching}. Our initial concentration event gives $d_G(v,P)=O(qd)$ for every $v$. The sets $X_i$ form a uniform equipartition, so another hypergeometric concentration estimate gives, simultaneously for all $v$ and $i$,
$d_G(v,X_i)\le 10qD$, with high probability. Since $q=c_0\delta$ and $c_0$ is sufficiently small, this is at most $\alpha D/20$. Similarly, the earlier concentration bounds give $d_G(v,A_1),d_G(v,B_1)\le\alpha D/20$ for every vertex. On the other hand, every vertex of every $X_i$, and every vertex of $A_1\cup B_1$, has at least $(1-\eta)D\ge D/2$ neighbours in each layer $C_i$.
\par By the probabilistic method, we may fix the $X_i$ to have the degree conditions required in the above paragraph. We can therefore apply Lemma~\ref{lem:perturbed-perfect-matching} successively, first with $H_0\to G[B_0,C_1]$ and $X,Y\to B_1,X_1$, and later applications with $H_0\to G[C_i,C_{i+1}]$ and $X,Y\to X_i,X_{i+1}$, and a final time with $H_0\to G[C_t,A_0]$ and $X,Y\to X_t,A_1$. In every case the hypotheses of Lemma~\ref{lem:perturbed-perfect-matching} hold, so there are perfect matchings
between
$B$ and $L_1$, between $L_i$ and $L_{i+1}$ for every $i<t$, and between $L_t$ and $A$.

The union of these matchings is a collection $\mathcal P$ of exactly $k$ vertex-disjoint paths, each starting in $B$ and ending in $A$. Since every intermediate matching is perfect, the internal vertices of these paths are exactly
$L_1\cup\cdots\cup L_t=V(G)\setminus V(S)$. Thus $\mathcal P$ is precisely the spanning path forest required by Proposition~\ref{prop:basic}. Combining $\mathcal P$ with the $A,B$-Hamilton router $S$ gives a Hamilton cycle in $G$. \end{proof}

\bibliographystyle{abbrv}
\bibliography{rbs}


\end{document}

%% file: guff.tex
\newcounter{capitalcounter}

\newcounter{claimcounter}

\newtheorem{lemma}{Lemma}[section]
\newtheorem{corollary}[lemma]{Corollary}
\newtheorem{theorem}[lemma]{Theorem}

\newtheorem{prop}[lemma]{Proposition}
\newtheorem{conjecture}[lemma]{Conjecture}

\theoremstyle{definition}
\newtheorem{defn}[lemma]{Definition}

\newtheorem{claim}{Claim}

\theoremstyle{remark}

\theoremstyle{plain}

\providecommand{\R}{\mathbb{R}}

\providecommand{\E}{\mathbb{E}}
\providecommand{\PP}{\mathbb{P}}

\providecommand{\rank}{\operatorname{rank}}

\providecommand{\norm}[1]{\left\lVert #1\right\rVert}

\allowdisplaybreaks[1]

%% file: rbs.bib
@incollection{HaggkvistThomason1997,
  author    = {H{\"a}ggkvist, Roland and Thomason, Andrew},
  title     = {Oriented {Hamilton} Cycles in Oriented Graphs},
  booktitle = {Combinatorics, Geometry and Probability: A Tribute to Paul Erd{\H{o}}s},
  editor    = {Bollob{\'a}s, B{\'e}la and Thomason, Andrew},
  publisher = {Cambridge University Press},
  address   = {Cambridge},
  year      = {1997},
  pages     = {339--354},
  doi       = {10.1017/CBO9780511662034.032}
}

@article{AlonHooryLinial,
  author  = {Alon, Noga and Hoory, Shlomo and Linial, Nathan},
  title   = {The {M}oore Bound for Irregular Graphs},
  journal = {Graphs and Combinatorics},
  volume  = {18},
  pages   = {53--57},
  year    = {2002},
  doi     = {10.1007/s003730200002}
}

@article{Merino2025Kneser,
  title   = {Kneser graphs are Hamiltonian},
  author  = {Merino, Arturo and M{\"u}tze, Torsten and Namrata},
  journal = {Advances in Mathematics},
  volume  = {468},
  pages   = {110189},
  year    = {2025},
  doi     = {10.1016/j.aim.2025.110189}
}

@inproceedings{AssadiKapralovYu2023,
  author    = {Sepehr Assadi and Michael Kapralov and Huacheng Yu},
  title     = {On Constructing Spanners from Random Gaussian Projections},
  booktitle = {Approximation, Randomization, and Combinatorial Optimization.
               Algorithms and Techniques (APPROX/RANDOM 2023)},
  series    = {Leibniz International Proceedings in Informatics (LIPIcs)},
  volume    = {275},
  pages     = {57:1--57:18},
  year      = {2023},
  publisher = {Schloss Dagstuhl -- Leibniz-Zentrum f{\"u}r Informatik},
  doi       = {10.4230/LIPIcs.APPROX/RANDOM.2023.57},
  url       = {https://doi.org/10.4230/LIPIcs.APPROX/RANDOM.2023.57}
}

@article{HooryLinialWigderson2006,
  author  = {Shlomo Hoory and Nathan Linial and Avi Wigderson},
  title   = {Expander Graphs and Their Applications},
  journal = {Bulletin of the American Mathematical Society},
  volume  = {43},
  number  = {4},
  pages   = {439--561},
  year    = {2006},
  doi     = {10.1090/S0273-0979-06-01126-8},
  url     = {https://www.cs.huji.ac.il/~nati/PAPERS/expander_survey.pdf}
}

@article{Tropp2012,
  author  = {Tropp, Joel A.},
  title   = {User-Friendly Tail Bounds for Sums of Random Matrices},
  journal = {Foundations of Computational Mathematics},
  volume  = {12},
  number  = {4},
  pages   = {389--434},
  year    = {2012},
  doi     = {10.1007/s10208-011-9099-z}
}

@article{Haxell1995,
  author  = {Penny E. Haxell},
  title   = {A Condition for Matchability in Hypergraphs},
  journal = {Graphs and Combinatorics},
  volume  = {11},
  number  = {3},
  pages   = {245--248},
  year    = {1995},
  doi     = {10.1007/BF01793010}
}

@inproceedings{AsadpourFeigeSaberi2008,
  author    = {Arash Asadpour and Uriel Feige and Amin Saberi},
  title     = {Santa Claus Meets Hypergraph Matchings},
  booktitle = {Approximation, Randomization and Combinatorial Optimization.
               Algorithms and Techniques},
  series    = {Lecture Notes in Computer Science},
  volume    = {5171},
  pages     = {10--20},
  publisher = {Springer},
  year      = {2008},
  doi       = {10.1007/978-3-540-85363-3_2}
}

@article{bowtell2026proof,
  title={A proof of {A}ndersen's rainbow path conjecture for large $ n$},
  author={Bowtell, Candida and Montgomery, Richard and M{\"u}yesser, Alp and Pokrovskiy, Alexey},
  journal={arXiv preprint arXiv:2608.06369},
  year={2026}
}

@book{BrouwerHaemers2012,
  author    = {Brouwer, Andries E. and Haemers, Willem H.},
  title     = {Spectra of Graphs},
  series    = {Universitext},
  publisher = {Springer},
  year      = {2012},
  doi       = {10.1007/978-1-4614-1939-6}
}

@article{CJMM,
  author  = {Chakraborti, Debsoumya and Janzer, Oliver
             and Methuku, Abhishek and Montgomery, Richard},
  title   = {Edge-disjoint cycles with the same vertex set},
  journal = {Advances in Mathematics},
  volume  = {469},
  pages   = {110228},
  year    = {2025},
  doi     = {10.1016/j.aim.2025.110228}
}

@article{LetzterMethukuSudakov2026,
  author  = {Letzter, Shoham and Methuku, Abhishek and Sudakov, Benny},
  title   = {Nearly {H}amilton cycles in sublinear expanders and applications},
  journal = {Journal of the London Mathematical Society},
  volume  = {113},
  number  = {2},
  pages   = {e70452},
  year    = {2026},
  doi     = {10.1112/jlms.70452}
}

@inproceedings{BucicHeHuangSaranurak2026,
  author    = {Buci{\'c}, Matija and He, Zhongtian and Huang, Shang-En and Saranurak, Thatchaphol},
  title     = {Disjoint Paths in Expanders in Deterministic Almost-Linear Time via Hypergraph Perfect Matching},
  booktitle = {Proceedings of the 2026 Annual ACM-SIAM Symposium on Discrete Algorithms (SODA)},
  pages     = {2316--2336},
  year      = {2026},
  doi       = {10.1137/1.9781611978971.83}
}

@article{HaggkvistThomason1995,
  author  = {H{\"a}ggkvist, Roland and Thomason, Andrew},
  title   = {Oriented {Hamilton} Cycles in Digraphs},
  journal = {Journal of Graph Theory},
  volume  = {19},
  number  = {4},
  pages   = {471--479},
  year    = {1995},
  doi     = {10.1002/jgt.3190190404}
}

@inproceedings{montgomery2026recent,
  title={Recent Progress in Graph Theory Using Expansion},
  author={Montgomery, Richard},
  booktitle={International Congress of Mathematicians 2026},
  pages={179--198},
  year={2026},
  organization={SIAM}
}

@article{bradavc2026hamiltonicity,
  title={Hamiltonicity of regular sublinear expanders},
  author={Brada{\v{c}}, Domagoj and Janzer, Oliver},
  journal={arXiv preprint arXiv:2605.15043},
  year={2026}
}

@article{krivelevich2003sparse,
  title={Sparse pseudo-random graphs are {H}amiltonian},
  author={Krivelevich, Michael and Sudakov, Benny},
  journal={Journal of Graph Theory},
  volume={42},
  number={1},
  pages={17--33},
  year={2003},
  publisher={Wiley Online Library}
}

@incollection{krivelevich2006pseudo,
  title={Pseudo-random graphs},
  author={Krivelevich, Michael and Sudakov, Benny},
  booktitle={More sets, graphs and numbers: A Salute to Vera Sos and Andr{\'a}s Hajnal},
  pages={199--262},
  year={2006},
  publisher={Springer}
}

@inproceedings{ChenKhannaLi22,
  author    = {Yu Chen and Sanjeev Khanna and Huan Li},
  title     = {On Weighted Graph Sparsification by Linear Sketching},
  booktitle = {63rd IEEE Annual Symposium on Foundations of Computer Science
               (FOCS 2022)},
  pages     = {474--485},
  publisher = {IEEE},
  year      = {2022},
  doi       = {10.1109/FOCS54457.2022.00052}
}

@article{ferber2025hamiltonicity,
  title={Hamiltonicity of sparse pseudorandom graphs},
  author={Ferber, Asaf and Han, Jie and Mao, Dingjia and Vershynin, Roman},
  journal={Combinatorics, Probability and Computing},
  volume={34},
  number={4},
  pages={596--620},
  year={2025},
  publisher={Cambridge University Press}
}

@article{hyde2025spanning,
  title={Spanning trees in pseudorandom graphs via sorting networks},
  author={Hyde, Joseph and Morrison, Natasha and M{\"u}yesser, Alp and Pavez-Sign{\'e}, Mat{\'\i}as},
  journal={Proceedings of the American Mathematical Society},
  volume={153},
  number={06},
  pages={2353--2367},
  year={2025}
}

@article{draganic2024hamiltonicity,
  title={Hamiltonicity of expanders: optimal bounds and applications},
  author={Dragani{\'c}, Nemanja and Montgomery, Richard and Correia, David Munh{\'a} and Pokrovskiy, Alexey and Sudakov, Benny},
  journal={arXiv preprint arXiv:2402.06603},
  year={2024}
}
